%% file: ota2invsets_core.tex
\documentclass[hidelinks,onefignum,onetabnum]{siamart220329}

\input{ota2invsets_shared}

\begin{document}

\maketitle

\begin{abstract}
	Any deterministic autonomous dynamical system may be globally linearized by its' Koopman operator.
	This object is typically infinite-dimensional and can be approximated by the so-called Dynamic Mode Decomposition (DMD).
	In DMD, the central idea is to preserve a fundamental property of the Koopman operator: linearity.
	This work augments DMD by preserving additional properties like functional relationships between observables and consistency along geometric invariants.
	The first set of constraints provides a framework for understanding DMD variants like Higher-order DMD and Affine DMD.
	The latter set guarantees the estimation of Koopman eigen-functions with eigen-value 1, whose level sets are known to delineate invariant sets.
	These benefits are realized with only a minimal increase in computational cost, primarily due to the linearity of constraints.
\end{abstract}



\section{Introduction}
Distilling a complex model into simpler and better interpretable formats is indispensable in science and engineering.
Among the numerous paths to achieve said model reduction, a particularly elegant approach is to leverage the associated Koopman operator \cite{koopman1931hamiltonian,mezic2005spectral}.
This object encodes any autonomous deterministic dynamical system as an equivalent evolution where the states (aka observables) are functions on the original state space.
Although the manipulation of observables often renders the Koopman operator to be infinite-dimensional, it is always linear.
Consequently, the operator can possess eigen-functions that, by definition, provide linear representations of the original dynamics.

The spectral objects associated to the Koopman operator can be approximated by the so-called Dynamic Mode Decomposition (DMD) \cite{rowley2009spectral,williams2015data}.
This algorithm produces a linear model (i.e., a matrix) that converges appropriately to the Koopman operator in the limit of infinite data \cite{korda2018convergence}.
In practice, data is often finite and the resulting DMD model may only be a \emph{quantitative approximation} of the underlying operator.
However, we might want \emph{qualitative exactness} in certain aspects. 
To paraphrase, we may wish the DMD model possess certain properties of the Koopman operator that are deemed important from a modeling perspective.

Some of these properties are induced by peculiarities of the underlying dynamical system.
For instance, measure-preserving dynamics makes the Koopman operator unitary \cite{koopman1931hamiltonian}.
Similarly, discrete symmetries induce a block-diagonal structure in the Koopman operator, when exhibited in an appropriate basis \cite{salova2019koopman}. 
In contrast, properties like positivity \cite{huang2018data,freeman2023data} and forward-backward consistency \cite{dawson2016__fb_TLS_DMD,azencot2019__consistent_DMD} are dynamics-independent i.e., are exhibited by the Koopman operator irrespective of the system that defines it. 

Regardless of their provenance, imbibing known properties of the Koopman operator permits the redirection of modeling resources towards its' under-resolved aspects.
The observed benefits of preserving such structure include greater computational efficiency in constructing the DMD model \cite{salova2019koopman} and enhanced estimation of Koopman spectral objects \cite{azencot2019__consistent_DMD,baddoo2021physics}. 
Some studies also report stable and accurate long-term forecasts \cite{baddoo2021physics,colbrook2023mpedmd}.

This work develops a variant of DMD dubbed \algoname~(\algonick) that guarantees consistency along known geometric and functional invariants.
\algonick~ uses linear constraints to enforce consistency along known fixed points, limit cycles and eigen-functions in the span of the observables chosen (\cref{ss:IC_DMD__Formulation}).
The simplicity of these requirements, in comparison to existing DMD variants, belies their conceptual and computational utility.

Conceptual insights include the (effective) recovery of Affine DMD \cite{hirsh2019centering} by consistently mapping the constant eigen-function (\cref{sss:Preserve_constant_EF__for__Affine_DMD}) and a re-formulation of Higher-order DMD \cite{le2017higher} by encoding the relationships between time-delayed observables (\cref{sss:Encode_delays__for__HO_DMD}).
Computational gains are centered around the estimation of Koopman spectral objects and forecasting.
Consistency along a fixed point or a limit cycle induces, by duality, a DMD eigen-function at 1 that, in theory, can be computed without solving an eigenvalue problem (\cref{ss:Induced_EFs}).

\section{Notation}\label{s:Notation}
We adopt a conventional notation for our primary objects of interest: vectors and matrices. 
The former are typically in \(\mathbb{C}^m\), labeled with bold-faced lower-case letters and represented as columns:
\begin{displaymath}
	\myvec{v} ~=~ [v_i ]_{i=1}^{m} ~=~ \begin{bmatrix}
		v_1 \\ v_2 \\ \vdots\\ v_m
	\end{bmatrix}.
\end{displaymath}
Constant vectors (those with only one unique coefficient) are represented by their value in boldface and dimension as a subscript.
\begin{displaymath}
	\mathbf{1}_m = [1]_{i=1}^{m}.
\end{displaymath}

Matrices are usually in \( \mathbb{C}^{m \times n}\), denoted by bold-faced upper-case letters and written as ordered sets of columns.
\begin{displaymath}
	\mymat{A} = [\myvec{a}_j ]_{j=1}^{n}  = \begin{bmatrix}
		\vline & \vline & &\vline\\
		\myvec{a}_1 & \myvec{a}_2 & \cdots & \myvec{a}_n \\
		\vline & \vline & &\vline
	\end{bmatrix} 
	= [a_{ij}]_{i,j=1}^{m,n}.
\end{displaymath}
Sub-scripts are used with bold-faced upper-case letters to describe the partitioning of a matrix, along columns or rows.
\begin{displaymath}
	\mymat{A}~=~
	\begin{bmatrix}
		\mymat{A}_1 & \mymat{A}_2 & \dots
	\end{bmatrix}
	~=~
	\begin{bmatrix}
		\rowselect{\mymat{A}}_1 \\
		\rowselect{\mymat{A}}_2 \\
		\vdots
	\end{bmatrix}.
\end{displaymath}
The (Hermitian) transpose of \(\mymat{A}\) is written as (\(\mymat{A}^{H}\)) \(\mymat{A}^T\).
The letter \(\mymat{I}\) is reserved for the identity matrix.
\begin{displaymath}
	\mymat{I}_{n} ~=~ [\myvec{e}_j]_{j=1}^{n} ~=:~ [\delta_{ij}]_{i,j=1}^n.
\end{displaymath}
Constant matrices are written similar to their vector counterparts:
\begin{displaymath}
	\mathbf{0}_{m \times n} = [\mathbf{0}_m]_{j=1}^{n} = [0]_{i,j=1}^{m,n}.
\end{displaymath}

Sub-spaces are written in script. For instance, \(\mathcal{R}(\mymat{Z})\) and \(\mathcal{N}(\mymat{Z})\) denote the column-space and null-space of \(\mymat{Z}\) respectively.
Orthogonal projectors play a crucial role in this work. 
Under the standard inner product on \(\mathbb{C}^m\), \(\mathcal{P}_{\mathcal{W}}\) denotes the orthogonal projector onto the subspace \(\mathcal{W}\).
If we let \(\mymat{A}^\dagger\) denote the Moore-Penrose pseudo-inverse of \(\mymat{A}\), we get:
\begin{displaymath}
	\mathcal{P}_{\mathcal{R}(\mymat{A})} ~=~\mymat{A} \mymat{A}^\dagger, \quad \mathcal{P}_{\mathcal{N}(\mymat{A}^H)} ~=~ \mymat{I} - \mymat{A} \mymat{A}^\dagger.
\end{displaymath}

\section{Background}

\subsection{Koopman operator}
%
Consider the autonomous discrete-time dynamical system defined by the update rule:
\begin{equation}\label{eq:DS4SDMD}
	\state^{+} = \Gamma (\state) , \; \state \in \mathbb{R}^\ssdim.
\end{equation}
Suppose we wish to find a coordinate transform that simplifies this dynamics, as formalized below:
\begin{definition}[Representation Eigen-Problem (REP)\cite{mezic2021koopman}]\label{defn:RepEig}
Find two functions \(f\) and \(\reduced{\Gamma}\) such that
	\begin{enumerate}
		\item The reduced dynamics is described by \(\reduced{\Gamma}\):
		\begin{displaymath}
			\forall~\state,\quad  f(\Gamma(\state)) = \reduced{\Gamma}(f(\state)).
		\end{displaymath}
		\item Evaluating \(\reduced{\Gamma} \circ f\) is computationally inexpensive compared to  \(f \circ \Gamma\).
	\end{enumerate}
\end{definition} 
This problem is not easy to solve, in part due to the innocuous requirement that \(\reduced{\Gamma}\) be a \textit{function}.
It restricts our choice of \(f\) to those functions whose every level-set is mapped by the dynamics only into another of its' level sets.
So, we can consider relaxing this restriction to permit \(\reduced{\Gamma}\) being an operator.
\begin{definition}[Relaxed Representation Eigen-Problem (Relaxed REP)]\label{defn:Relaxed_RepEig}
	Find a function \(f\) and an operator \(\reduced{\Gamma} \) that meet all the requirements listed in \cref{defn:RepEig}
\end{definition}
Conceptually, we are permitting \(f(\Gamma(\state))\), the future value of \(f\) at \(\state\), to be forecast using the entire function, \(f\), instead of just its' value at \(\state\), \(f(\state)\):
\begin{displaymath}
	\forall~\state,\quad f(\Gamma(\state)) = (\reduced{\Gamma} \circ f)(\state).
\end{displaymath}
Such a massive increase in model capacity usually requires a proportionate hike in computational resources.
Although the resolution of \cref{defn:Relaxed_RepEig} against this backdrop ( find \(f\) such that \(\reduced{\Gamma} \circ f \) is inexpensive, despite the exorbitant cost associated with \(\reduced{\Gamma}\)) may seem bleak, an elegant solution exists nonetheless and begins with just a closer look.

First, we need to recognize that the set of all feasible operators for \cref{defn:Relaxed_RepEig} contains only a unique element - the Koopman operator.
\begin{definition}[Koopman operator \cite{koopman1931hamiltonian}]\label{defn:Koopman_operator}
	Let \(\psi\) denote any complex-valued function on \(\mathbb{R}^\ssdim\).
	Then, the Koopman operator \(U\) associated to \cref{eq:DS4SDMD} is defined as follows:
	\begin{equation}
		U \circ \psi := \psi \circ \Gamma.
	\end{equation}
\end{definition}
Hence, the Koopman operator meets the first condition of Relaxed REP by construction.

Now, even if \cref{eq:DS4SDMD} describes a non-linear phenomenon in a finite-dimensional vector space, the dynamics it induces through \(U\), on a (typically) infinite-dimensional vector space of functions, is \textit{linear}.
More precisely, the Koopman operator, \(U\), is linear and this impels us to search for its' eigen-functions and eigen-values.
\begin{definition}[Koopman eigen-functions]
	The function \(\phi\) is said to be an eigen-function of the Koopman operator \(U\) with eigen-value \(\lambda\) if and only if the following holds:
	\begin{equation}
		U \circ \phi  = \lambda ~\phi.
	\end{equation}
\end{definition}
To paraphrase, the Koopman operator \(U\) acts on the eigen-function \(\phi\) by simply scaling it with its' associated eigen-value \(\lambda\).
The triviality of this update renders \(U \circ \phi \) vastly more economical than \(\phi \circ \Gamma \). 
Therefore, the pair \((U, \phi)\) is a solution to the Relaxed Representation Eigen-Problem (\ref{defn:Relaxed_RepEig}).

Furthermore, by invoking  \cref{defn:Koopman_operator}, we can also establish that the eigen-function \(\phi\) along with the operation of scaling by \(\lambda\) resolves the Representation Eigen-Problem (\ref{defn:RepEig}) as well.
\begin{displaymath}
	\begin{aligned}
		f~&=~\phi, \quad \quad \quad \forall z \in \mathbb{C},\quad \reduced{\Gamma}(z) = \lambda~z \\
		\implies f(\Gamma(\state))  ~=~ (\phi \circ \Gamma)(\state) ~&=~ (U \circ \phi)(\state) ~=~ \lambda~\phi(\state) ~=~ \reduced{\Gamma}(\phi(\state))
		~=~ \reduced{\Gamma}(f(\state)).
	\end{aligned}
\end{displaymath}

Finally, the linearity of \(U\) causes any function in the span of Koopman eigen-functions to mildly inherit the simplicity of its' constituents. 
The Koopman operator updates such a function by simply scaling the eigen-functions present in its' expansion.
\begin{definition}[Koopman Mode Decomposition \cite{mezic2005spectral}]\label{defn:Koopman_Mode_Decomposition}
		Suppose the function \(f\) lies in the span of the Koopman eigen-functions \(\{\phi_i\}_{i \in \mathcal{I}}\) with eigen-values \(\{\lambda_i\}_{i \in \mathcal{I}}\).
		Then, there exist numbers \(\{c_i\}_{i\in \mathcal{I}}\) such that
		\begin{displaymath}
			f ~=~ \sum_{i \in \mathcal{I}} c_i~\phi_i.
		\end{displaymath}
		The quantities \(\{c_i\}_{i\in \mathcal{I}}\) are termed the Koopman modes of \(f\) associated with eigen-values \(\{\lambda_i\}_{i \in \mathcal{I}}\).
		In addition, for any whole number \(n\), the action of the Koopman operator \(U\) is given by a super-position of its' actions on the constituent eigen-functions.
		\begin{displaymath}
			U^n \circ f ~=~ U^n \circ \left( \sum_{i \in \mathcal{I}} c_i~\phi_i \right)
			~=~ \sum_{i \in \mathcal{I}} c_i~ \left( U^n \circ  \phi_i \right)
			~=~ \sum_{i \in \mathcal{I}} c_i~ \lambda_i^n \phi_i .
		\end{displaymath} 
\end{definition}

Apart from the functional payoffs seen so far, Koopman eigen-functions also connect to geometric aspects of the underlying dynamical system, like isochrons and invariant sets.
Since the latter provide a testbed for illustrating \algoname, we briefly outline the essentials before describing the numerical algorithm (\naiveDMDAcronym) that serves as the foundation for \algonick.

\subsubsection{Invariant sets and the Koopman connect}
An invariant set is any subset of the state-space that is mapped into itself by the dynamics, in both forward and backward time.
\begin{definition}[Invariant set]\label{defn:Invariant set}
	A set \(I \subseteq \mathbb{R}^\ssdim\) is an invariant set of the dynamical system \cref{eq:DS4SDMD} if and only if the image and pre-image under \(\Gamma\) of every point in \(I\) remains inside \(I\).
	\begin{displaymath}
		I~\textrm{is invariant}~~~\iff~~~ \forall~ \state \in I,~~\Gamma[\state] \cup \Gamma^{-1}[\state] ~\subseteq~ I.
	\end{displaymath}
\end{definition}
Put differently, an invariant set is a dynamically self-sufficient entity.
The dynamics within an invariant set \(I\) can be understood using just the restriction of the update rule to itself  \((\Gamma|_I)\), disregarding the evolution occurring elsewhere \((\Gamma|_{I^c})\).
Hence, the notion of invariance transmutes set-theoretic disjointedness into dynamical decoupling.
If two disjoint sets \(A\) and \(B\) are also invariant, then the dynamics within each set \((A ~\textrm{or}~ B)\) can be analyzed oblivious of the happenings in the other \((B ~\textrm{or}~A)\).
As such, when the invariant sets in question are bounded and the update rule permits economical surrogates, set invariance can translate to significant computational efficiencies.

The dynamical autarkies discussed above are connected to the Koopman spectrum at 1 \cite{mezic1994geometrical}:
\begin{theorem}\label{thm:Invariance_via_Koopman}
	\(I\) is an invariant set of \cref{eq:DS4SDMD} if and only if \(I\) is a level set of some Koopman eigen-function with eigen-value 1.
\end{theorem}
Thus, we can estimate invariant sets by equivalently approximating Koopman eigen-functions whose corresponding eigen-value is 1.
%

\subsection{Extended Dynamic Mode Decomposition (\naiveDMDAcronym)}
\naiveDMDFlavour\DMD is an algorithm used to estimate Koopman spectral objects \cite{williams2015data}.
\naiveDMDAcronym~ approximates the action of the Koopman operator, \(U\), on a user-specified set of functions (termed the dictionary).
The approximation is realized as a matrix (\(\naiveDMDMatrix\)), and its' eigen-decomposition yields the desired spectral information.

Suppose we observe the system through \(\dictionarylength\) functions \(\dictionary~:=~ [\psi_i]_{i=1}^\dictionarylength \).
Let \(\dictionary\) and its image under the Koopman operator, \(U \circ \dictionary\), be sampled at \(\trainlength\) points \( \{\state_j\}_{j=1}^\trainlength \) to generate the data matrices \(\mymat{X}\) and \(\mymat{Y}\).
\begin{equation}\label{eq:Data_4_Exact_Extended_DMD}
	\mymat{X} := [\dictionary(\state_j)]_{j=1}^\trainlength,\quad  
	\mymat{Y} := [(U \circ \dictionary)(\state_j)]_{j=1}^\trainlength = [\dictionary (\Gamma(\state_j))]_{j=1}^\trainlength .
\end{equation}
Then, an approximation to the action of \(U\) on \(\dictionary\) is given by any solution of the following least-squares regression:
\begin{equation}\label{eq:DMD_ObjectiveFunction}
	\naiveDMDMatrix \in \arg\min_{\mymat{A}}~ \lVert \mymat{A} \mymat{X} - \mymat{Y} \rVert_{F}^2 .
\end{equation}

The eigen-values and eigen-vectors of \(\naiveDMDMatrix\) approximate Koopman spectral quantities.
If \((\myvec{w}, \lambda)\) is a left eigen-pair of \(\naiveDMDMatrix\) i.e.
\begin{displaymath}
	\myvec{w}^H \naiveDMDMatrix~=~ \lambda \myvec{w}^H,
\end{displaymath}
then, the function \(\myvec{w}^H\dictionary\) approximates a Koopman eigen-function with eigenvalue \(\lambda\).
Similarly, if \((\myvec{v}, \lambda)\) is a right eigen-pair of \(\naiveDMDMatrix\) i.e.
\begin{displaymath}
	\naiveDMDMatrix \myvec{v} ~=~ \lambda \myvec{v},
\end{displaymath}
then, \(\myvec{v}\) approximates the Koopman mode corresponding to eigenvalue \(\lambda\), of the dictionary \(\dictionary\).

\section{Novelty statement}\label{s:Novelty_Statement}
\subsection{Some shortcomings of \naiveDMDAcronym}
Suppose \(\myvec{w}^H \dictionary\) is a Koopman eigen-function with eigen-value \(\lambda\).
\begin{equation}\label{eq:Assertion_KoopamEigenFunction}
	U \circ (\myvec{w}^H \dictionary) ~=~ \lambda~ \myvec{w}^H \dictionary
\end{equation}
Then, it is reasonable to expect that \cref{eq:Assertion_KoopamEigenFunction} is reflected by \(\naiveDMDMatrix\) as
\begin{equation}\label{eq:Expectation__EDMD_EigenFunction}
	\myvec{w}^H \naiveDMDMatrix ~=~ \lambda~ \myvec{w}^H.
\end{equation} 
Although this property is not guaranteed by the definition of \naiveDMDAcronym~ in \cref{eq:DMD_ObjectiveFunction}, it holds whenever \(\mymat{X}\) has full row rank.
This verifiable criterion necessitates linearly-independent observables\footnote{\(\myvec{c}^H \dictionary = 0 \iff \myvec{c} = 0.\)} and at-least as many samples as observables \((\trainlength \geq \dictionarylength)\).
The former serves as a rationale for projecting \(\mymat{X}\) and \(\mymat{Y}\) onto the span of \(\mymat{X}\), before using them as inputs for \cref{eq:DMD_ObjectiveFunction} \cite{schmid2010dynamic}.
The latter is a de-facto standard when performing \naiveDMDAcronym~ on low-dimensional systems, and in associated convergence guarantees.
While the strategies above can remedy the violation of \cref{eq:Expectation__EDMD_EigenFunction}, they are nonetheless extrinsic to the \naiveDMDAcronym~ formulation \cref{eq:DMD_ObjectiveFunction}.
This can limit their efficacy in the dual setting that follows.

Say \(\opt{\state}\) is a fixed point of the dynamical system \cref{eq:DS4SDMD}.
Then, we have
\begin{displaymath}
	(U \circ \dictionary)[\opt{\state}] ~=~ \dictionary[\opt{\state}].
\end{displaymath}
So, it seems natural to ask that \(\naiveDMDMatrix\) imbibe this property as well i.e. we want
\begin{equation}\label{eq:Expectation__EDMD_FixedPoint}
	\naiveDMDMatrix~ \dictionary[\opt{\state}] ~=~ \dictionary[\opt{\state}].
\end{equation}
Alas, this requirement can remain unmet even if observables are linearly independent and \(\trainlength >> \dictionarylength\).
To see this, suppose our observations of \cref{eq:DS4SDMD} have limited resolution.
Then, it may be that there exists a point \(\dummy{\state}\) such that 
\begin{displaymath}
	\dictionary[\opt{\state}] ~=~ \dictionary[\dummy{\state}] ~\neq~ \dictionary[\Gamma[\dummy{\state}]].
\end{displaymath}
If both \(\opt{\state}\) and \(\dummy{\state}\) are incorporated within the data matrix \(\mymat{X}\), then there is no linear map that can transform \(\mymat{X}\) into \(\mymat{Y}\).
However, the optimal way to attempt this impossible task defines \(\naiveDMDMatrix\) (see \cref{eq:DMD_ObjectiveFunction}).
Hence, the \naiveDMDAcronym~ model can fail to consistently map the fixed point \(\opt{\state}\) i.e. the condition \cref{eq:Expectation__EDMD_FixedPoint} can be violated. 

\subsection{Resolution by \algonick}
\algoname~ intrinsically satisfies the modeling requirements \cref{eq:Expectation__EDMD_EigenFunction,eq:Expectation__EDMD_FixedPoint} by appending them as constraints to the \naiveDMDAcronym~ objective in \cref{eq:DMD_ObjectiveFunction}.
\begin{enumerate}
	\item The resulting formulation provides insights into two established variants of \DMD: Affine DMD \cite{hirsh2019centering} and Higher-order DMD \cite{le2017higher}.
	\item Linearity of constraints endows \algonick~ with an explicit solution whose computational cost remains comparable to \naiveDMDAcronym.
	\item Finally, consistent mapping of fixed points permits the construction of \algonick~ eigen-functions corresponding to eigen-value 1, \textit{without solving the associated eigen-value problem}.
\end{enumerate}

\section{Literature survey}
\algoname~ incorporates knowledge of functional and geometric invariants into \naiveDMDAcronym, thereby producing estimates of invariant sets.
Each of these salient aspects is surveyed below.

\subsection{Structure preserving improvements of DMD}


Dynamics-independent properties of the Koopman operator were the first to be incorporated into DMD models.
Forward-backward DMD \cite{dawson2016__fb_TLS_DMD} and it's upgrade Consistent DMD \cite{azencot2019__consistent_DMD} use the relationship between the spectra of \(U\) and \(U^{-1}\) to simultaneously reduce the bias and variance of DMD eigenvalues.
Naturally-structured DMD \cite{huang2018data} preserves the positivity of \(U\) and the Markov property of \(U^*\), thus guaranteeing a DMD eigen-function at 1.

Even so, assimilating dynamics-induced aspects of \(U\) into DMD has emerged an active area of research.
Salova and co-workers \cite{salova2019koopman} use discrete symmetries to implant Koopman invariant subspaces into reasonably generic dictionaries. This procedure block-diagonalizes the associated DMD matrix and, thereby, reduces the computational cost of its' construction.
PI-DMD \cite{baddoo2021physics} studies the effect of incorporating a plethora of physics into DMD, and exhibits superior Koopman spectral estimation relative to vanilla DMD in all associated case-studies.
mpEDMD \cite{colbrook2023mpedmd} rigorously extends PI-DMD to measure-preserving dynamics and reports greater accuracy in long-term forecasts.

The theoretical insights from \algonick~ (\cref{ss:DMD_Connections}) find their roots in dynamics-independent properties of \(U\), while the computational gains (\cref{ss:Induced_EFs}) largely arise from dynamics-induced aspects.


%


\subsection{Estimation of invariant sets}
The common thread connecting techniques to estimate invariant sets is the approximation of a function that is known to possess (some) invariant pre-images.
For example, Lyapunov-based approaches construct an energy-like function, that is non-increasing along every trajectory, and return its' domain as the estimate \cite{giesl2015review}.
Similarly, the Koopman operator perspective embodied by \cref{thm:Invariance_via_Koopman} suggests finding eigen-functions with eigen-value 1, and then perusing their level sets.

Initial work leveraging the Koopman view-point focused on measure-preserving dynamics and was based on time-averaging \cite{mezic1999method}.
The variational alternative that is \naiveDMDAcronym~can be much more sample-efficient, provided the dictionary is Koopman-invariant.
Unfortunately, this criterion is rarely easy to meet and its' violation can preclude \(\naiveDMDMatrix \) from possessing an eigen-function with eigen-value 1 .
Nonetheless, we gain reasonable insights from the eigen-function whose eigen-value is closest to 1.
This stopgap remedy to extracting invariant sets from \(\naiveDMDMatrix \), without relying on the condition of Koopman invariance, has been improved upon in many ways.
We discuss a few in the order of increasing computational cost.

Garcia-Tenorio and co-workers \cite{garcia2022evaluation} construct a unitary eigen-function from the \naiveDMDAcronym~ approximant, using the algebraic closure of Koopman eigen-functions under products and non-negative powers.
This post-processing has a negligible cost, so building \(\naiveDMDMatrix \) remains the computationally intensive step.
Huang and Vaidya \cite{huang2018data} inoculate \(\naiveDMDMatrix\) with the positivity of the Koopman operator along with the Markov property of its' adjoint, the Perron-Frobenius operator.
The resulting model, dubbed Naturally Structured DMD, is guaranteed to have an eigen-value at 1.
Computing the associated matrix representation requires solving a second-order cone program, which is a palatable jump from the quadratic solve required of \naiveDMDAcronym.
Otto and Rowley \cite{otto2019linearly} explicitly require their DMD model of the Duffing oscillator to contain the number 1 in its spectrum. 
Consequently, they obtain a well-resolved unitary eigen-function of the Duffing oscillator.
But, building this model necessitates solving a highly nonlinear optimization problem, making it much more expensive than the two preceding developments.

\algonick~ guarantees the estimation of Koopman eigen-functions with eigen-value 1, and hence invariant sets, at a cost comparable to \naiveDMDAcronym.

\begin{remark}
	 A complementary approach to estimating invariant sets involves the Perron-Frobenius operator, where the cynosure is once again an eigen-function with eigen-value 1 \cite{dellnitz2002set}.
	 The estimation here is philosophically closer to the time-averaging method \cite{mezic1999method} in that the eigen-function is estimated \textit{directly}, in a recursive manner, without resorting to any intermediaries like the \(\naiveDMDMatrix\) matrix.
\end{remark}

\section{\algoname}
\subsection{Formulation}\label{ss:IC_DMD__Formulation}
Suppose we have some states and functions
\begin{displaymath}
	\{\state^{\rm geom}_i\}_{i=1}^\geomconscount,\quad  \{ \fconvec_j^H\dictionary \}_{j=1}^{\funconscount}
\end{displaymath}
 that should be exactly mapped to their image under the dynamics as viewed via \(\dictionary\):
 \begin{displaymath}
 	\{\Gamma[\state^{\rm geom}_i]\}_{i=1}^\geomconscount, \quad  \{ U \circ \left( \fconvec_j^H\dictionary \right)\}_{j=1}^{\funconscount}.
 \end{displaymath}
Fixed-points are examples of the former, and the latter can include Koopman eigen-functions present in the span of \(\dictionary\).

This additional requirement can be precisely quantified when the functions in purview are mapped by \(U\) into the span of \(\dictionary\), i.e.,
\begin{displaymath}
	\forall~ j = 1:\funconscount,\quad \exists~ \fconplusvec_j ~\textrm{such that}~   U \circ \left( \fconvec_j^H\dictionary \right) ~=~ (\fconplusvec_j)^H\dictionary,
\end{displaymath}
as we can then construct the following matrices:
\begin{equation}
	\gcon ~:=~ \left[~~\dictionary[~\state^{\rm geom}_i~]~~\right]_{i=1}^\geomconscount ,\quad 
	\gconplus~:=~ \left[~~\dictionary[~\Gamma[\state^{\rm geom}_i]~]~~\right]_{i=1}^\geomconscount.
\end{equation}
\begin{equation}
		\fcon ~:=~ \left[\fconvec_j\right]_{j=1}^{\funconscount}  , \quad
		\fconplus ~:=~ \left[\fconplusvec_j\right]_{j=1}^{\funconscount}.
\end{equation}
Now, our ask can be seen to requiring the DMD model \(\mymat{A}\) lie in the affine subspace,
\begin{equation}\label{eq:defn__IC_DMD_Constraints}
	\mathcal{A}~:=~\{\mymat{A}~|~\mymat{A} \gcon~=~\gconplus~\&~\fcon^H \mymat{A}~=~(\fconplus)^H\}.
\end{equation}
The model represented by \(\mymat{A}_{\rm Exact}\) \cref{eq:DMD_ObjectiveFunction} may not always respect the above requirement.
Nonetheless, this shortcoming can be easily ameliorated by simply selecting
an element of \(\mathcal{A}\) that best maps \(\mymat{X}\) onto \(\mymat{Y}\):
\begin{equation}\label{eq:IC_DMD}
	\optA_{\rm IC} \in \arg\min_{\mymat{A} \in \mathcal{A} }~ \lVert \mymat{A} \mymat{X} - \mymat{Y} \rVert_{F}^2 .
\end{equation}
The above minimization defines \algoname~(\algonick)~ and  \(\optA_{\rm IC}\) will alternatively be called the \algonick~ matrix.
\textit{One} way of choosing \(\optA_{\rm IC}\) is described by the following theorem:
\begin{theorem}\label{thm:Astar_IC_is_a_minimizer}
	Suppose the constraints defining \(\mathcal{A}\) are compatible and possess no redundancies i.e.
	\begin{subequations}\label{eq:Assumptions_4_IC_DMD}
		\begin{align}
			\textrm{Compatibility:}&~~~~	\fcon^H \gconplus~=~(\fconplus)^H \gcon. \label{eq:Assumption_Compatibility}\\
			\textrm{No redundancies:}&~~~~	\gcon~ \mathrm{and}~ \fcon ~\textrm{have full column rank}.		\label{eq:Assumption_No_Redunancies}
		\end{align}
	\end{subequations}
	Let the matrices \(\optC_0\) and \(\optA_{\rm lsq}\) be defined as follows:
	\begin{subequations}
		\begin{align}
			\optC_0~&:=~\gconplus \gcon^\dagger + \left( \fcon^\dagger \right)^H (\fconplus)^H \gcon^\perp \left( \gcon^\perp \right)^H \label{eq:defn_C0_star},\\
			\optA_{\rm lsq}~&:=~ \left(\fcon^\perp\right)^H (\mymat{Y} - \optC_0 \mymat{X}) \left( \left( \gcon^\perp \right)^H \mymat{X} \right)^\dagger \label{eq:defn_Astar_lsq}.
		\end{align}
	\end{subequations}
	Then, the matrix 
	\begin{displaymath}
		\optC_0 + \fcon^\perp  	\optA_{\rm lsq}  \left( \gcon^\perp \right)^H
	\end{displaymath}
	is a minimizer of \cref{eq:IC_DMD}, and hence constitutes a valid choice for \(\optA_{\rm IC}\).
\end{theorem}
\begin{proof}
	Refer \cref{ss:Constructing__Astar_IC}.
\end{proof}

\subsection{Connections to existing DMD variants}\label{ss:DMD_Connections}
\subsubsection{Affine DMD implicitly preserves the constant eigen-function}\label{sss:Preserve_constant_EF__for__Affine_DMD}
Let the constant function be the first observable in our dictionary \(\dictionary\) i.e.:
\begin{displaymath}
	\psi_1 = 1 \implies
	\mymat{X}~=~
	\begin{bmatrix}
		\myvec{1}_n^H \\ \, \\
		\rowselect{\mymat{X}}_2
	\end{bmatrix},
\quad
	\mymat{Y} ~=~
	\begin{bmatrix}
		\myvec{1}_n^H \\ \, \\
		\rowselect{\mymat{Y}}_2
	\end{bmatrix}.
\end{displaymath}
It is trivial to see that \(\psi_1\) is a Koopman eigenfunction with eigenvalue 1.

Suppose we perform DMD only using the observables in \(\dictionary\) other than \(\psi_1\).
This amounts to modeling \(\rowselect{\mymat{Y}}_2\) as the image of \(\rowselect{\mymat{X}}_2\) under a linear map.
Affine DMD \cite{hirsh2019centering} augments that prediction with an eponymous term, and its' model parameters are defined by the following minimization routine:
\begin{equation}\label{eq:defn_Affine_DMD}
	\reduced{\mymat{A}}, \reduced{\myvec{b}} :=
	\begin{cases}
		\arg \min_{\mymat{A}, \myvec{b}}~ \left\lVert \mymat{A} \right\rVert_F^2,  &\mathcal{R}\left(\rowselect{\mymat{Y}}_2^T\right) \subseteq \mathcal{R}\left(\mymat{X}^T\right) \\
		\textrm{subject to}\quad \mymat{A} \rowselect{\mymat{X}}_2 + \myvec{b} \myvec{1}_n^H = \rowselect{\mymat{Y}}_2 & \\ \\
		\arg \min_{\mymat{A}, \myvec{b}}~ \left\lVert \mymat{A} \rowselect{\mymat{X}}_2 + \myvec{b} \myvec{1}_n^H - \rowselect{\mymat{Y}}_2
		\right \rVert_{F}^2, &\mathcal{R}\left(\rowselect{\mymat{Y}}_2^T\right) \nsubseteq \mathcal{R}\left(\mymat{X}^T\right) 
	\end{cases}
\end{equation}
Regardless of which case holds true, Hirsch and co-workers established \cite{hirsh2019centering}, among other things, that the sample means  
\begin{equation}\label{eq:defn_Affine_DMD__Sample_means}
	\myvec{\mu}_{\mymat{X}} ~:=~ \rowselect{\mymat{X}}_2 \frac{\myvec{1}_\trainlength}{\trainlength}, \quad 
	\myvec{\mu}_{\mymat{Y}} ~:=~ \rowselect{\mymat{Y}}_2 
	\frac{\myvec{1}_{\trainlength}}{\trainlength},
\end{equation}
connect  \(\reduced{\mymat{A}}\) and \(\reduced{\myvec{b}}\) as below: 
\begin{equation}\label{eq:Affine_DMD_offset_from_rests}
	\reduced{\myvec{b}} ~=~ \myvec{\mu}_{\mymat{Y}} - \reduced{\mymat{A}} \myvec{\mu}_{\mymat{X}}.
\end{equation}

We will see that Affine DMD is effectively recovered by \algonick~ on the complete set of observables \(\dictionary\), provided the constant eigen-function is consistently mapped and no geometric invariant is encoded. 
In particular, the Affine DMD model from \cref{eq:defn_Affine_DMD} can always be used to construct a valid \algonick~ model i.e. a minimizer of \cref{eq:IC_DMD}.
Furthermore,  \cref{eq:Affine_DMD_offset_from_rests}, which is a relation between Affine DMD model parameters, is analogously observed in every solution to \cref{eq:IC_DMD}.

\begin{proposition}\label{prop:Affine_DMD__recovery}
	Suppose we make the following choices:
	\begin{equation}\label{eq:Constraints_for_Affine_DMD}
		\fcon~=~\fconplus~=
		\begin{bmatrix}
		1 \\ \myvec{0}_{n-1}
		\end{bmatrix}, \quad 
	\gcon~=~\gconplus~=~\mymat{0}.
	\end{equation}
	Then, the following is true:
	\begin{enumerate}
		\item Every minimizer of the \algonick~ minimization \cref{eq:IC_DMD} shares the same first row: 
		\begin{displaymath}
			\optA_{\rm IC} ~=~
			\begin{bmatrix}
				1 &  \\
				\left(\optA_{\rm IC}\right)_{21} & \left(\optA_{\rm IC}\right)_{22}
			\end{bmatrix}.
		\end{displaymath}
		\item The matrix,
		\begin{displaymath}
			\optA_{\rm Affine}~:=~
			\begin{bmatrix}
				1 & \\
				\reduced{\myvec{b}} & \reduced{\mymat{A}}
			\end{bmatrix},
		\end{displaymath}
		is a solution to \cref{eq:IC_DMD}.
		\item Every minimizer of \cref{eq:IC_DMD} also satisfies the following analogue of \cref{eq:Affine_DMD_offset_from_rests}:
		\begin{equation}\label{eq:IC_DMD__offsets_analogoues_to__Affine_DMD}
			\left(\optA_{\rm IC}\right)_{21} ~=~\myvec{\mu}_{\mymat{Y}} - \left( \optA_{\rm IC}\right)_{22} \myvec{\mu}_{\mymat{X}}.
		\end{equation}
	\end{enumerate}
\end{proposition}
\begin{proof}
	Refer \cref{ss:Proof__Affine_DMD__recovery}.
\end{proof}

\subsubsection{Higher-order DMD amounts to assimilating the relationships between delay-embedded observables}\label{sss:Encode_delays__for__HO_DMD}

Suppose \(\dictionary\), which is a vector of \(\dictionarylength\) observables, is constructed by taking \(\fundelayscount\) time delays of  a smaller observable vector, \(\undelayedictionary\), of length \(\undelayedictlength\).
\begin{equation}\label{eq:psi_from_delays_of_psibreve}
	\dictionary~=~
	\begin{bmatrix}
		\undelayedictionary \\
		U \circ	\undelayedictionary\\
		\vdots \\
		U^{\fundelayscount} \circ \undelayedictionary 
	\end{bmatrix}.
\end{equation}
Then, the corresponding DMD model, dubbed Higher-order DMD (HO-DMD), is ideally parametrized by only \(d+1\) reduced matrices \( \left\{ \left(\mymat{A}\right)_j \right\}_{j=1}^{d+1} \) and, more crucially, possesses the following block-Companion structure  \cite{le2017higher}:
	\begin{displaymath}
	\mymat{A}_{\rm HO-DMD}~=~
	\begin{bmatrix}
		& \mymat{I}_{\undelayedictlength} & & \\[0.3cm]
		& & \ddots & \\[0.4cm]
		& & & \mymat{I}_{\undelayedictlength}\\[0.2cm]
		\left(\mymat{A}\right)_1 & \left(\mymat{A}\right)_2 & \hdots  & \left(\mymat{A}\right)_{\fundelayscount+1}
	\end{bmatrix}.
\end{displaymath} 

It is easy to show that encoding the dependencies among delay-embedded observables within the \algonick~ framework also induces the above block-Companion structure.
\begin{proposition}\label{prop:HO_DMD__connect}
Suppose the observables in \(\dictionary\) are generated by delay-embedding a smaller dictionary \(\undelayedictionary\) as described in \cref{eq:psi_from_delays_of_psibreve}.	
If we let
	\begin{equation}\label{eq:Constraints_for_HO_DMD}
		\fcon~=~
	\begin{bmatrix}
		\mymat{I}_{\undelayedictlength \fundelayscount} \\
		~
	\end{bmatrix}\quad \textrm{and} \quad
	\fconplus~=~
	\begin{bmatrix}
		\\
		\mymat{I}_{\undelayedictlength \fundelayscount}
	\end{bmatrix},
\end{equation}
then every IC-DMD model defined by \cref{eq:IC_DMD} has the following block Companion structure:
	\begin{displaymath}
	\optA_{\rm IC}~=~
	\begin{bmatrix}
		& \mymat{I}_{\undelayedictlength} & & \\[0.3cm]
		& & \ddots & \\[0.4cm]
		& & & \mymat{I}_{\undelayedictlength}\\[0.2cm]
		\left(\reducedoptA_{\rm IC}\right)_1 & \left(\reducedoptA_{\rm IC}\right)_2 & \hdots  & \left(\reducedoptA_{\rm IC}\right)_{\fundelayscount+1}
	\end{bmatrix}.
\end{displaymath}
The sub-matrices \( \left\{\left(\reducedoptA_{\rm IC}\right)_j\right\}_{j=1}^{d+1} \) define \(\optA_{\rm IC}\) as described by \cref{eq:IC_DMD}.
\end{proposition}


\subsection{Duality-induced eigen-functions}\label{ss:Induced_EFs}
Suppose there is a non-zero complex number \(\lambda\) and a matrix \(\gconequalizer\) with linearly independent columns such that
\begin{equation}\label{eq:defn_gconequalizer}
	\gconplus\gconequalizer~=~\lambda~ \gcon \gconequalizer.
\end{equation}
This scenario is possible if at least one column of \(\gcon\) represents a fixed point or if a subset of the columns fully encodes limit cycle.
For instance, if \(\gcon\) is produced by sampling \(\dictionary\) at a bunch of fixed points, then \cref{eq:defn_gconequalizer} is satisfied for  \(\widehat{\mymat{V}} = \mymat{I}\) and \(\lambda~=~1\).
Similarly, if the columns of \(\gcon\) correspond to sampling a limit cycle, then \cref{eq:defn_gconequalizer} holds for  \(\widehat{\mymat{V}} = \myvec{1}\) and \(\lambda~=~1\).
When both fixed points and limit cycles are present, \(\widehat{\mymat{V}}\) will select all fixed points and average each limit cycle.

The columns of the matrix, 
\begin{equation}\label{eq:defn_gcon_0}
	\gcon_0~:=~\gcon \gconequalizer,
\end{equation}
are linearly independent right eigen-vectors of \(\optA_{\rm IC}\) with eigenvalue \(\lambda\), i.e.,
\begin{displaymath}
	\optA_{\rm IC} \gcon_0~=~ \lambda~ \gcon_0.
\end{displaymath}
The linear independence is due to the full column rank of \(\gconequalizer\) and \(\gcon\), while the eigen-relation falls from \cref{eq:IC_DMD} which says that \(\optA_{\rm IC} \in \mathcal{A}\).
Consequently, by duality, there exists a full column rank matrix \(\mymat{W}\) such that 
\begin{subequations}\label{eq:defn_Induced_EFs}
	\begin{align}
		\mymat{W}^H \optA_{\rm IC}~&=~ \lambda \mymat{W}^H \label{eq:Induced_EFs_are_EFs}\\
		\mymat{W}^H \gcon_0~&=~\mymat{I} . \label{eq:Induced_EF_reps_are_dual_2_gcon0}
	\end{align}
\end{subequations}
In other words, the functions represented by the columns of \(\mymat{W}\), viz. \(\{\myvec{w}_i^H \dictionary\}_i \), are approximate eigen-functions of the Koopman operator with eigenvalue \(\lambda\).
These objects, henceforth termed \emph{induced eigen-functions} to highlight their provenance, crystallize the benefits of encoding geometric invariants in \algonick.
Specifically, \algonick~ enjoys the following two advantages that are difficult to replicate in \naiveDMDAcronym:
\begin{enumerate}
	\item The linear system \cref{eq:defn_Induced_EFs} always has a solution when the assumptions used to generate \(\optA_{\rm IC},~\gcon_0~\textrm{and}~\lambda\) are met.
	In comparison, even the linear system  \cref{eq:Induced_EFs_are_EFs}, let alone \cref{eq:defn_Induced_EFs}, is unlikely to have an exact solution if \(\optA_{\rm IC}\)  were replaced with \(\naiveDMDMatrix\).
	\item 	Knowledge of the eigen-value \(\lambda\) and the duality condition \cref{eq:Induced_EF_reps_are_dual_2_gcon0} enables, in theory, the computation of the associated induced eigen-functions \emph{without an explicit eigen-decomposition}. 
	This stands in sharp contrast to the analogous procedure for finding an eigen-function with eigen-value \(\lambda\) from an \naiveDMDAcronym~ model: Perform an eigen-decomposition of \(\naiveDMDMatrix\), locate the eigen-value closest to \(\lambda\) and peruse the associated eigen-function.
\end{enumerate}
\begin{remark}
	In practice, solving \cref{eq:Induced_EFs_are_EFs,eq:Induced_EF_reps_are_dual_2_gcon0} simultaneously yields a numerically inferior estimate of \(\mymat{W}\) than the sequential approach of solving \cref{eq:Induced_EFs_are_EFs} followed by \cref{eq:Induced_EF_reps_are_dual_2_gcon0}.
	Hence, we adopt the latter strategy in the upcoming computational studies.
\end{remark}

\section{Numerical experiments}


We visualize induced eigen-functions, as defined in \cref{eq:defn_Induced_EFs} and associated with an eigen-value \(\lambda ~=~1\), for simple low-dimensional examples\footnote{The pertinent code can be found at \url{https://github.com/gowtham-ss-ragavan/IC-DMD.git}}.
Each example possesses the following common attributes: 
\begin{enumerate}
	\item The underlying discrete-time system is generated by observing an autonomous ordinary differential equation (ODE) at equi-spaced times.
	In particular, if the ODE, 
	\begin{displaymath}
		\dot{\state} ~=~ \velocity(\state),
	\end{displaymath}
	possesses the flow-map \(\flow\), i.e.
	\begin{displaymath}
		\flow(0, \state)~=~ \state, \quad (\partial_1\flow)(t, \state)~=~ \velocity(\flow(t, \state)),
	\end{displaymath}
	then, we pick some positive real number \(k\) and choose the update law \(\Gamma\) in \cref{eq:DS4SDMD} to be
	\begin{displaymath}
		\Gamma(\state) ~:=~ \flow(k, \state).
	\end{displaymath}
	\item All fixed points and limit cycles of \cref{eq:DS4SDMD} lie well within the \(\ssdim-\)dimensional hyper-cube \([-1,1]^\ssdim\). 
	Hence, we only sample this subset of \(\mathbb{R}^\ssdim\), in an equi-spaced fashion, to generate \(\mymat{X}\).
	\item The \algonick~ model encodes every fixed point and limit cycle, along with the constant function\footnote{Which is a trivial eigen-function with eigen-value 1.}.
	The sole exception to this choice is the very first study (see \cref{fig:1DStable_Comparison}) which probes the differential gain from specifying geometric invariants.
	\item The choice of dictionary always ensures \cref{eq:Assumptions_4_IC_DMD} holds. 
	So, we recover as many induced eigen-functions as the number of geometric invariants encoded.
\end{enumerate}

We begin by observing the multi-stable one-dimensional ODE,
\begin{equation}\label{eq:1D_Stable}
	\dot{\state} = -\left( \state + \frac{1}{2} \right) \left( \state - \frac{2}{10} \right) \left( \state - \frac{7}{10} \right),
\end{equation}
at time-steps of \(k = 0.1\).
The resulting discrete-time system has three hyperbolic fixed points at \( \opt{\state} = -\frac{1}{2},~\frac{2}{10}~\textrm{and}~\frac{7}{10}\).
The fixed point at \(\frac{2}{10}\) is unstable while the other two are stable.

The utility of consistently mapping the three fixed points can be explored by incrementally adding in the associated constraints to \algonick.
This can be done in three steps: Begin with an \naiveDMDAcronym~ model, add in the constraint associated to the constant eigen-function and then encode the three fixed points.
Each of these three models can generate approximate Koopman eigen-functions with eigen-value 1.
For the first model, we choose the \naiveDMDAcronym~ eigen-functions whose eigen-values are closest to 1 and then search for a solution to \cref{eq:Induced_EF_reps_are_dual_2_gcon0} in their span.
For the remaining two models, eigen-function estimates are obtained by sequentially solving \cref{eq:defn_Induced_EFs}.
Overall, we get three approximate Koopman eigen-functions for each model.
All of these are visualized in \cref{fig:1DStable_Comparison}, where the models are distinguished by color and grouped by the fixed point that is theoretically supposed, by \cref{eq:Induced_EF_reps_are_dual_2_gcon0}, to lie in the \(1-\)level set.
\begin{figure}[]
	\centering
	\subfloat[\(x^* = -\frac{1}{2} \)]{\includegraphics[width = 0.48 \linewidth]{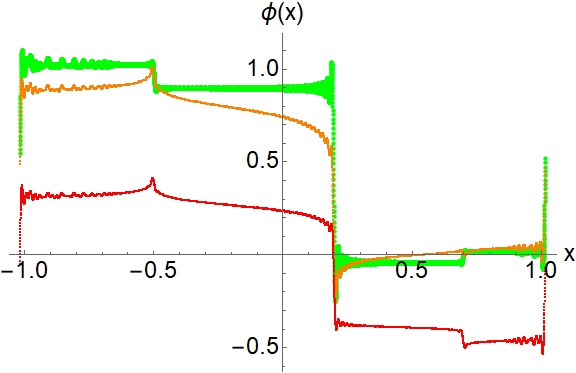}} \quad 
	\subfloat[\(x^* = \frac{2}{10}\)]{\includegraphics[width = 0.48 \linewidth]{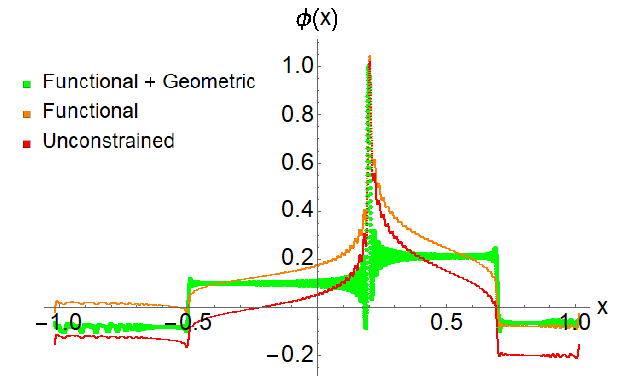}} \\
	\subfloat[\(x^* = \frac{7}{10}\)]{\includegraphics[width = 0.48 \linewidth]{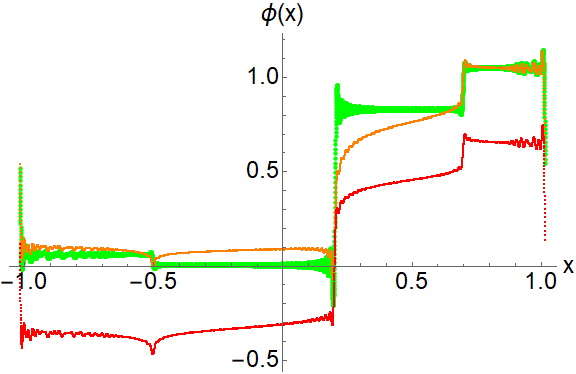}}
	\caption{Unitary eigen-functions for the discrete-time system obtained by observing the ODE \cref{eq:1D_Stable} with time-steps of \(k = 0.1\).
		In each plot, every curve depicts the best approximation to a Koopman eigen-function with eigen-value 1, using a specific DMD model.
		The red curves are obtained from \naiveDMDAcronym, by effectively solving \cref{eq:defn_Induced_EFs} using \(\naiveDMDMatrix\) instead of \(\optA_{\rm IC}\).
		The orange curves are induced eigen-functions from \algonick, when the constant eigen-function is consistently mapped, but without encoding any fixed point.
		The green curves are an upgrade of their orange counterparts, when all three fixed points are also specified in \algonick.
		Common to all models is a dictionary of trigonometric functions and around \(10,000\) equi-spaced training points to generate \(\mymat{X}\) and \(\mymat{Y}\).
		Every figure caption highlights a point that is theoretically expected to be in the \(1\)-level set of each visualized function.
		The utility of these functions is quantified by the extent to which their level sets are invariant.
		Each approximation manages to capture set invariance at-least locally.
		However, the best global performance is seen in the fully-informed \algonick~ model (colored in green).
		}
	\label{fig:1DStable_Comparison}
\end{figure}
Any approximation of a Koopman eigen-function with eigen-value 1 must, by \cref{thm:Invariance_via_Koopman}, take a constant value over every invariant set.
The \naiveDMDAcronym~ model, shown in red, demonstrates a reasonable agreement to the left of \(-\frac{1}{2}\) and to the right of \(\frac{7}{10}\).
But, this performance deteriorates over the invariant sets \(\left[-\frac{1}{2}, \frac{2}{10}\right]\) and \(\left[\frac{2}{10}, \frac{7}{10}\right]\), as particularly seen in panel (b).
These flaws seem to persist in the \algonick~ model that only encodes the constant eigen-function (shown in orange).
On the upshot, it seems to satisfy \cref{eq:Induced_EF_reps_are_dual_2_gcon0} by reducing to the appropriate indicator function upon restriction to the subset \(\{-\frac{1}{2},~\frac{2}{10},~\frac{7}{10}\}\).
This benefit is retained by the \algonick~ model that also encodes the fixed points (shown in green).
More importantly, it also does a much better job at ensuring a constant value over invariant sets.

Despite the differences between these three models, there is a common signature of the stability possessed by each fixed point.
The variation exhibited (by an induced eigen-function) across any fixed point seems indicative of its' stability.
In every panel of \cref{fig:1DStable_Comparison}, the stable fixed points, \(\opt{\state} = -\frac{1}{2},~\frac{7}{10}\), see much smaller changes when compared to the unstable fixed point \(\opt{\state} = \frac{2}{10}\).

In addition, we note that our choice of dictionary is sub-optimal for representing induced eigen-functions.
In particular, we've been using smooth functions to approximate what appears to be a piecewise discontinuous function.
Hence, for the forthcoming studies, we switch to observables that can better accommodate discontinuities.
Specifically, we choose indicator functions whose supports uniformly partition the hyper-cube \([-1,1]^\ssdim\).

The benefits of switching to a discontinuity-friendly dictionary materialize in \cref{fig:1DStable}.
Here, we visualize the induced eigen-functions obtained by propagating the dictionary update onto the \algonick~ model that encodes all the fixed points.
\begin{figure}[]
	\centering
	\subfloat[\(x^* = -\frac{1}{2}\)]{\includegraphics[width = 0.48 \linewidth]{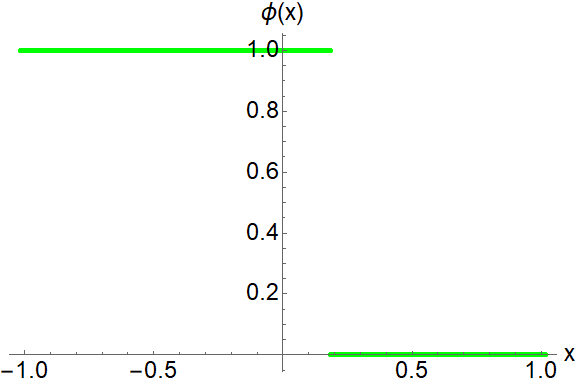}} \quad 
	\subfloat[\(x^* = \frac{2}{10}\)]{\includegraphics[width = 0.48 \linewidth]{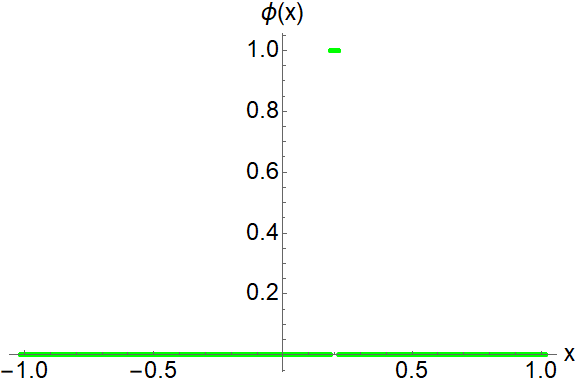}} \\
	\subfloat[\(x^* = \frac{7}{10}\)]{\includegraphics[width = 0.48 \linewidth]{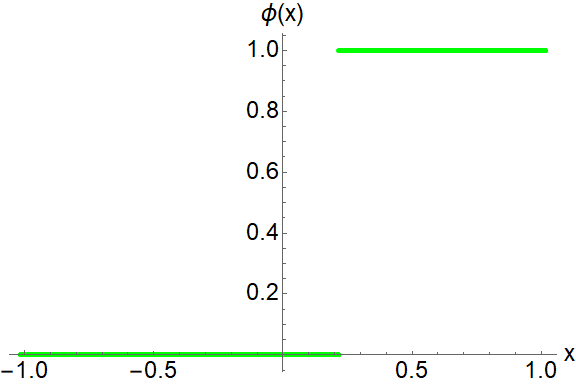}}
	\caption{Unitary eigen-functions for the discretization of \cref{eq:1D_Stable} with a time-step of \(k = 0.1\).
	These are obtained by re-building the geometrically consistent \algonick~ model from \cref{fig:1DStable_Comparison}, using discontinuity-friendly observables.
	Specifically, the new dictionary comprises of around 60 indicator functions.
	Their supports partition the interval \([-1,1]\) into sub-intervals of equal length. 
	The training data comprises of nearly 12,000 equi-spaced samples on \([-1,1]\).
	Evidently, indicator functions neatly resolve the discontinuities possessed by unitary Koopman eigen-functions, across different invariant sets.
	Indeed, this particular dictionary seems to correlate the stability of every fixed point with continuity at the same point and the size of any encapsulating non-zero level set.
	}
	\label{fig:1DStable}
\end{figure}
Within every invariant set, all three eigen-functions are practically constant.
Furthermore, they reflect the stability of each fixed point much better than in \cref{fig:1DStable_Comparison}.
Every eigen-function is \textit{continuous} across the stable fixed points \(\opt{\state}~=~-\frac{1}{2},~\frac{7}{10}\), and \textit{discontinuous} at the unstable fixed point \(\opt{\state}~=~\frac{2}{10}.\)

All three fixed points in \cref{eq:1D_Stable} are hyperbolic.
We can check if this commonality contributes to the stability signature discussed above, by studying the following modification of \cref{eq:1D_Stable}:
\begin{equation}\label{eq:1D_HalfStable}
	\dot{\state} = \left(\state + \frac{1}{2}\right) \left(\state - \frac{2}{10}\right)^2 \left(\state - \frac{7}{10}\right).
\end{equation}
Even though we have the same fixed points,  \(\opt{\state} = \frac{2}{10}\) is non-hyperbolic.
The other two remain hyperbolic, but \(\opt{\state} = \frac{7}{10}\) is now unstable.
The only stable fixed point is \(\opt{\state} = -\frac{1}{2}.\)
\Cref{fig:1DHalfStable} shows the induced eigen-functions produced by encoding all fixed points and the constant function in \algonick.
\begin{figure}[]
	\centering
	\subfloat[\(x^* = -\frac{1}{2}\)]{\includegraphics[width = 0.48 \linewidth]{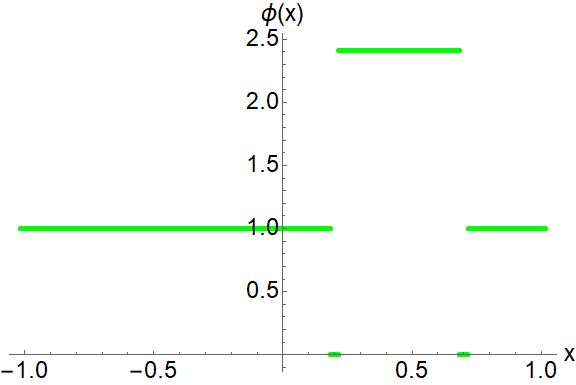}} \quad 
	\subfloat[\(x^* = \frac{2}{10}\)]{\includegraphics[width = 0.48 \linewidth]{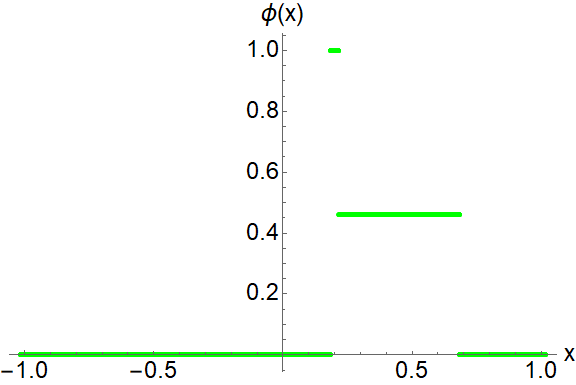}} \\
	\subfloat[\(x^* = \frac{7}{10}\)]{\includegraphics[width = 0.48 \linewidth]{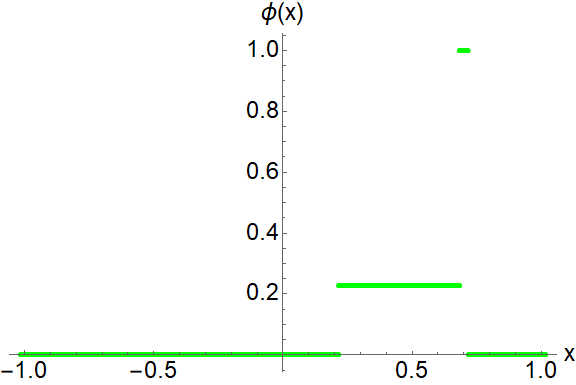}}
	\caption{Unitary eigen-functions of \cref{eq:1D_HalfStable}, when the latter is observed over time-steps of \(k = 0.1\).
	The \algonick~ model is built with the same parameters used for \cref{fig:1DStable}.
	In particular, all three fixed points and the constant eigen-function are encoded as constraints.
	The level sets of induced eigen-functions remain invariant, even-though the fixed point at \(0.2\) is now non-hyperbolic.
	As before, continuity and the size of any encapsulating non-zero level set appear to be indicators of fixed point stability.
	In panel (a), the \(1-\) level set contains the stable fixed point at \(-\frac{1}{2}\) in its' interior and has a non-zero length.
	In contrast, its' counterparts in the other panels contain unstable fixed points, are effectively jump discontinuities, and so possess almost vanishing length.
}

	\label{fig:1DHalfStable}
\end{figure}
The induced eigen-functions continue to correctly depict set-invariance.
The intervals \((-\infty, -0.5]\), \( [-0.5, 0.2] \), \([0.2, 0.7]\) and  \([0.7, \infty)\) are invariant.
On each of these sets, every induced eigen-function remains constant.
Moreover, its' continuity at any fixed point continues to be indicative of the stability ascribed to the latter.

We now proceed to two dimensional systems, starting with the Duffing oscillator described by the following ODE:
\begin{equation}\label{eq:Duffing}
	\begin{aligned}
		\dot{\upsilon}_1 &= \upsilon_2, \\
		\dot{\upsilon}_2 &= -\upsilon_2 + \upsilon_1- 36~\upsilon_1^3.
	\end{aligned}
\end{equation}
We study the discrete-time system obtained by sampling \cref{eq:Duffing} at time steps of \(k = 1.6\).
This system has three fixed points, all of which lie on the line \(\upsilon_2~=~0\).
The fixed points at \(\opt{\upsilon}_1~=~\pm \frac{1}{6}\) are stable, while the one at \(\opt{\upsilon}_1~=~0\) is a saddle point.
In \cref{fig:Duffing}, we portray the induced eigen-functions generated by the \algonick~model that encodes all three fixed points and the constant function.
\begin{figure}[]
	\centering
	\subfloat[\(x^* = -\frac{1}{6}\)]{\includegraphics[width = 0.48 \linewidth]{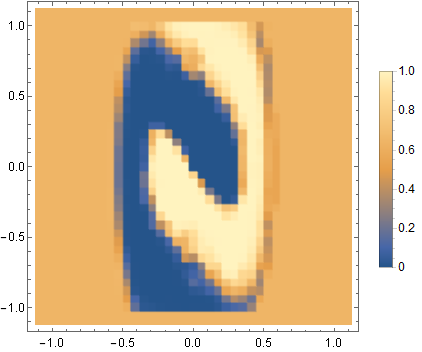}} \quad 
	\subfloat[\(x^* = 0\)]{\includegraphics[width = 0.48 \linewidth]{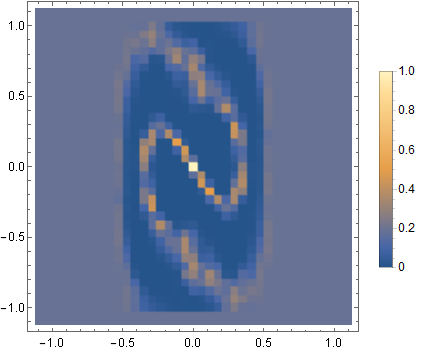}} \\
	\subfloat[\(x^* = \frac{1}{6}\)]{\includegraphics[width = 0.48 \linewidth]{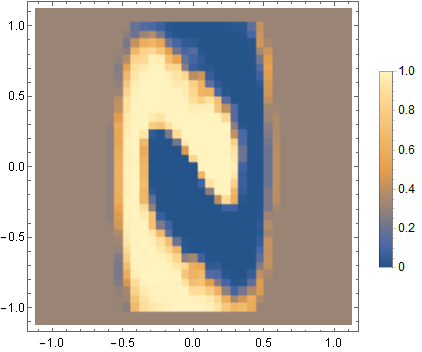}}
	\caption{Unitary eigen-functions* of the Duffing oscillator \cref{eq:Duffing}, seen under a temporal resolution of \(k = 1.6\).
	They are represented using a dictionary of 1225 indicator functions that are two-dimensional adaptations of those used in \cref{fig:1DStable,fig:1DHalfStable}.
	The training data is made of nearly 31,000 initial conditions.
	The \algonick~ model encodes all three fixed points, and the constant eigen-function.
	In each panel, the level set @ 1 approximates the basin-of-attraction associated with the fixed point it contains.
	Consequently, in panel (b), the \(1\) level set highlights the stable manifold of the saddle point located at the origin.
	}
	\label{fig:Duffing}
\end{figure}
Each stable fixed point possesses a basin-of-attraction, and these comprise the only non-trivial invariant sets in this system.
Within each of these sets, we note that every eigen-function (in \cref{fig:Duffing}) is constant, at least within the region \(|\upsilon_1| \leq 0.5\).
Inside this subset, we also see that the level set at \(1\) in panel (b) sketches the stable manifold of the saddle point at the origin.
Regrettably, set invariance is not accurately portrayed for \(|\upsilon_1| > 0.5\).
This could be due to the two invariant sets being so intertwined in this region, that resolving them is beyond the choice of observables and training data.

Fixed point stability continues to be correlated with induced eigen-functions.
In the one-dimensional systems studied earlier, continuity of an induced eigen-function at a fixed point was indicative of stability.
This argument can be generalized to higher dimensions, by checking continuity along all rays emanating from a fixed point.
However, a simpler alternative is to consider the measure possessed by any non-zero level set that contains the fixed point.
In panels (a) and (c) of \cref{fig:Duffing}, we see that the level set at \(1\) contains the stable fixed point at \(\opt{\upsilon}_1 = \mp \frac{1}{6}\) respectively, and has a non-zero measure.
In comparison, the level set at \(1\) in panel (b) contains the unstable fixed point at the origin, and has near-zero measure.
Hence, there is a numerical correlation between the stability of a fixed point and the measure of any non-zero level set (of an induced eigen-function with eigen-value 1) containing it.

We close our computational studies with a system admitting multiple limit cycles.
We consider the two-dimensional ODE written in polar coordinates,
\begin{displaymath}
	\upsilon_1~=~r \cos(\theta),\quad \upsilon_2~=~r \sin(\theta),
\end{displaymath}
as
\begin{equation}\label{eq:CompetingLimitCycles}
	\begin{aligned}
	\dot{r} &= r\left(r-\frac{1}{3}\right)\left(r - \frac{2}{3}\right), \\
	 \dot{\theta} &= 2 \pi,
	\end{aligned}
\end{equation}
and observe it at uniform time steps of \(k = \frac{1}{6}\).
The resulting discrete-time dynamics admits an unstable fixed point at the origin and a continuum of stable (unstable) period-6 orbits at \(r = \frac{1}{3}\) \((r = \frac{2}{3})\).
We compute an \algonick~ model incorporating the fixed point and 4 periodic orbits each from \(r = \frac{1}{3}\) and \(r = \frac{2}{3}\).
The resulting induced eigen-functions are grouped by their provenance and visualized in \cref{fig:2D_LimitCycles}.
\begin{figure}[]
	\centering
	\subfloat[\(r^* = 0\)]{\includegraphics[width = 0.48 \linewidth]{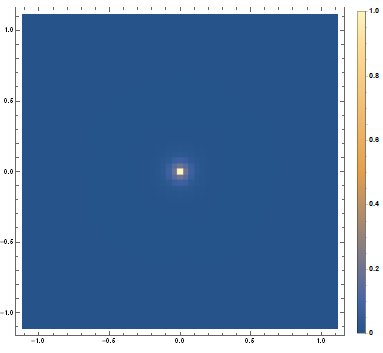}} \quad 
	\subfloat[\(r^* = 1/3\)]{\includegraphics[width = 0.48 \linewidth]{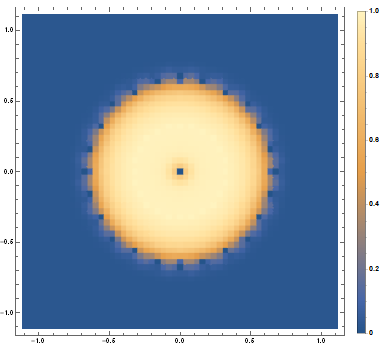}} \\
	\subfloat[\(r^* = 2/3\)]{\includegraphics[width = 0.48 \linewidth]{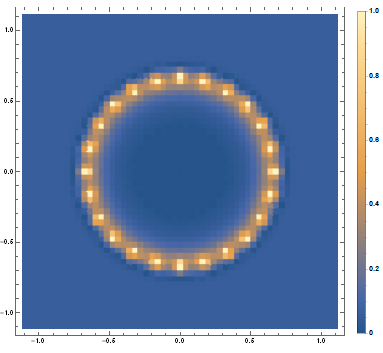}}
	\caption{Unitary eigen-functions associated to \cref{eq:CompetingLimitCycles} when sampled at time-steps of \(k = 1/6\).
	The observables are refinements of the indicator functions used in \cref{fig:Duffing}, and number around 2600.
	The \algonick~ model, which uses close to 13,000 initial conditions as training data, encodes the fixed point at the origin, both limit cycles and the constant eigen-function.
	As before, the stability of a geometric invariant correlates with the size of whichever non-zero level set contains it.
	In addition, the level set @ 1 in panel (b) delineates the basin-of-attraction corresponding to the stable limit cycle at \(r = 1/3\).
	}
	\label{fig:2D_LimitCycles}
\end{figure}
The annular regions demarcated by the circles \(r = 0,~\frac{1}{3}\) and \(r = \frac{2}{3}\) are invariant.
Over each of these sets, the induced eigen-functions are constant.
In addition, the stability of a fixed point or limit cycle remains correlated with the measure of whichever non-zero level set encapsulates it.

\section{Conclusions and future work}
\algoname~ illustrates the benefits to integrating knowledge of invariants into \naiveDMDAcronym.
Two seemingly disparate variants of DMD, Affine DMD and Higher Order DMD, are seen to possess a common theme: consistent encoding of eigen-functions and functional dependencies within a dictionary.
The same principle applied in a geometric context generates, with less hassle than \naiveDMDAcronym, estimates for Koopman eigen-functions with eigen-value 1.
Their level sets are approximately invariant and, hence, of great practical import as they enable a reductionist analysis of nonlinear dynamics.

There are foundational and algorithmic aspects of \algoname~ that need to be understood and improved upon.
For instance, linearity of constraints is crucial for computational efficiency and the ability to estimate invariant sets.
However, the same property limits the extent to which functional constraints are integrated into the \algonick~ model.
For example, suppose we know that \(\phi\) is a Koopman eigen-function with eigen-value \(\lambda\).
Assume that the observables \(\{ \phi^i \}_{i=1}^\genericcount\) are all present in the span of the dictionary \(\dictionary\).
Then, we need to encode each of the \(\genericcount\) eigen-functions within \cref{eq:IC_DMD} \textit{individually} - a particularly cumbersome task if the observables are monomials.
This calls for a modification of \algonick~ so that specifying the eigen-function \(\phi\) automatically encodes the remaining \(\genericcount-1\) eigen-functions.

On an algorithmic note, estimating invariant sets via \algonick~ appears to be data-intensive, even for the low-dimensional problems considered here.
While this may be partly attributed to its' origins in \naiveDMDAcronym, we cannot neglect that we've been trying to either approximate a discontinuous function using smooth observables or integrate discontinuous functions.
Regardless, the approximation of Koopman eigen-functions with eigen-value 1 makes the case for considering dictionaries that can well-represent discontinuous functions.

Finally, estimating an induced eigen-function with eigen-value 1 requires accurate knowledge of the fixed point. 
This requirement can be a bottleneck, since a fixed point is usually found by a nonlinear solve and, hence, of limited resolution.
An easy solution could be to augment the associated geometric constraint with a consistent perturbation, and include the same deviation as a decision variable.
But, this leads to a bilinear program which is far less-tractable than a quadratic program.
Hence, there is a need for assimilating uncertainty while maintaining a comparable computational cost.




\section*{Acknowledgments}

This work was supported by the Army Research Office (ARO-MURI W911NF-17-1-030), the National Science Foundation
(Grant no. 1935327) and a UC Regents in ME fellowship (2017-18).

\appendix
\section{Proofs}
\subsection{On the construction of \(\optA_{\rm IC}\)}\label{ss:Constructing__Astar_IC}

The key idea used to prove \cref{thm:Astar_IC_is_a_minimizer} is that an explicit parametrization of the feasible set \(\mathcal{A}\) renders \cref{eq:IC_DMD} equivalent to an ordinary least-squares problem.
To this end, we pick two matrices \(\gcon^\perp\) and \(\fcon^\perp\)  whose columns comprise an orthonormal basis for \(\mathcal{N}\left( \gcon^H \right)\) and \(\mathcal{N}\left(\fcon^H\right)\) respectively.
In other words, \((\gcon^\perp,~\fcon^\perp)\) is any pair of full column rank matrices that satisfies the following identities:
\begin{subequations}
	\begin{align}
		\mymat{I}~&=~\gcon \gcon^\dagger ~+~ \gcon^\perp \left( \gcon^\perp \right)^H  \label{eq:I_equals_range_plus_null_for_D}\\
		\mymat{I}~&=~\fcon \fcon^\dagger ~+~ \fcon^\perp \left( \fcon^\perp \right)^H. \label{eq:I_equals_range_plus_null_for_E}
	\end{align}
\end{subequations}
The afore-mentioned parametrization can now be formalized.
\begin{lemma}\label{lem:Parametrization_mathcal_A}
	The feasible set of \algonick, defined in \cref{eq:defn__IC_DMD_Constraints} and denoted as \(\mathcal{A}\), is the affine subspace generated by adding an offset of \(\optC_0\) to  the range of the linear map \(\tilde{\mymat{A}} \rightarrow \fcon^\perp \tilde{\mymat{A}} \left( \gcon^\perp \right)^H\).
	\begin{equation}\label{eq:Parametrization_of_mathcal_A}
		\mathcal{A}~=~	\{\mymat{A}~|~\exists ~\tilde{\mymat{A}}~\textrm{such that}~ \mymat{A}~=~\optC_0 + \fcon^\perp \tilde{\mymat{A}} (\gcon^\perp)^H  \}.
	\end{equation}
\end{lemma}
\begin{proof}
	We will show that the two sets in question are subsets of each other.
	
	Suppose \(\mymat{A} \in \mathcal{A}\).
	We can show that there exists a \(\tilde{\mymat{A}}\) with the desired properties by expanding \(\mymat{A}\) in a basis of the matrix vector-space that it belongs to.
	We begin by pre-multiplying \cref{eq:I_equals_range_plus_null_for_D} with \(\mymat{A}\) and then using the relation \(\mymat{A} \gcon~=~\gconplus\), we get:
	\begin{equation}\label{eq:A_after_AQ_R}
		\mymat{A}
		~=~ \mymat{A}\gcon \gcon^\dagger ~+~ \mymat{A}\gcon^\perp \left( \gcon^\perp \right)^H
		~=~ \gconplus \gcon^\dagger ~+~ \mymat{A}\gcon^\perp \left( \gcon^\perp \right)^H
		.
	\end{equation}
	In order to remove \(\mymat{A}\) from the second term, we use the fact that orthogonal projectors, like \(\fcon \fcon^\dagger\), are Hermitian to rewrite \cref{eq:I_equals_range_plus_null_for_E}:
	\begin{displaymath}
		\mymat{I}
		~=~ \left(\fcon^\dagger\right)^H \fcon^H ~+~ \fcon^\perp \left( \fcon^\perp \right)^H.
	\end{displaymath}	
	Post-multiplying both sides with \(\mymat{A}\) and using the relation \( \fcon^H \mymat{A} = (\fconplus)^H\), we get:
	\begin{displaymath}
		\mymat{A}
		~=~ \left(\fcon^\dagger\right)^H \fcon^H \mymat{A} ~+~ \fcon^\perp \left( \fcon^\perp \right)^H \mymat{A} 
		~=~ \left(\fcon^\dagger\right)^H (\fconplus)^H  ~+~ \fcon^\perp \left( \fcon^\perp \right)^H \mymat{A}. 
	\end{displaymath}	
	Substituting the final expression for \(\mymat{A}\) in the term \(\mymat{A}\gcon^\perp \left( \gcon^\perp \right)^H\) of \cref{eq:A_after_AQ_R}, we bring \(\optC_0\) into the picture.
	\begin{displaymath}
		\begin{aligned}
			\mymat{A}
			~&=~ \gconplus \gcon^\dagger 
			~+~ 
			\left(
			\left(\fcon^\dagger\right)^H (\fconplus)^H  ~+~ \fcon^\perp \left( \fcon^\perp \right)^H \mymat{A} 
			\right)
			\gcon^\perp \left( \gcon^\perp \right)^H \\
			~&=~ \optC_0 + \fcon^\perp \left( \fcon^\perp \right)^H \mymat{A} \gcon^\perp \left( \gcon^\perp \right)^H
			.
		\end{aligned}
	\end{displaymath}
	If we define 
	\begin{displaymath}
		\tilde{\mymat{A}}~:=~  \left( \fcon^\perp \right)^H \mymat{A} \gcon^\perp,
	\end{displaymath}
	then we obtain
	\begin{displaymath}
		\mymat{A}~=~\optC_0 + \fcon^\perp \tilde{\mymat{A}} \left( \gcon^\perp \right)^H.
	\end{displaymath}
	Hence, we have shown the following set containment:
	\begin{displaymath}
		\mathcal{A}~\subset~	\{\mymat{A}~|~\exists ~\tilde{\mymat{A}}~\textrm{such that}~ \mymat{A}~=~\optC_0 + \fcon^\perp \tilde{\mymat{A}} (\gcon^\perp)^H  \}.
	\end{displaymath}
	
	To complete the proof, we still need to prove that the set containment above also holds in the opposite direction.
	Consider any matrix \(\mymat{A}\) for which there exists \(\tilde{\mymat{A}}\) satisfying 
	\begin{equation}\label{eq:A_explicit_check}
		\mymat{A}~=~\optC_0 + \fcon^\perp \tilde{\mymat{A}} (\gcon^\perp)^H.
	\end{equation}
	Post-multiplying the above by \(\gcon\) and using the definition of \(\gcon^\perp\), we find:
	\begin{displaymath}
		\mymat{A}\gcon
		~=~\optC_0 \gcon + \fcon^\perp \tilde{\mymat{A}} 
		\underbrace{(\gcon^\perp)^H \gcon}
		~=~ \optC_0 \gcon.
	\end{displaymath}
	This can be further simplified using \cref{eq:defn_C0_star} and the linear independence of the columns of \(\gcon\) \cref{eq:Assumption_No_Redunancies}:
	\begin{displaymath}
		\optC_0 \gcon
		~=~
		\gconplus \underbrace{\gcon^\dagger \gcon} + \left( \fcon^\dagger \right)^H (\fconplus)^H \gcon^\perp  
		\underbrace{\left( \gcon^\perp \right)^H \gcon}
		~=~ \gconplus
	\end{displaymath}
	Hence, \(\mymat{A} \) satisfies
	\begin{displaymath}
		\mymat{A} \gcon~=~\gconplus.
	\end{displaymath}
	Verifying the other equality defining \(\mathcal{A}\) proceeds similarly by pre-multiplying \cref{eq:A_explicit_check} with \(\fcon^H\).
	\begin{displaymath}
		\fcon^H	\mymat{A}
		~=~ 
		\fcon^H	\optC_0 + 
		\underbrace{\fcon^H	\fcon^\perp} \tilde{\mymat{A}} (\gcon^\perp)^H
		~=~ 
		\fcon^H	\optC_0.
	\end{displaymath}
	Once again, we invoke the definition of \(\optC_0\) alongside  \(\fcon\)'s full column rank to get:
	\begin{displaymath}
		\begin{aligned}
			\fcon^H	\optC_0
			~&=~ \fcon^H \gconplus \gcon^\dagger + 
			\underbrace{\fcon^H \left( \fcon^\dagger \right)^H} (\fconplus)^H \gcon^\perp \left( \gcon^\perp \right)^H \\
			~&=~ \fcon^H \gconplus \gcon^\dagger + 
			(\fconplus)^H \gcon^\perp \left( \gcon^\perp \right)^H.
		\end{aligned}
	\end{displaymath}
	The two terms above can be merged using the compatibility condition in \cref{eq:Assumption_Compatibility} and the identity \cref{eq:I_equals_range_plus_null_for_D}:
	\begin{displaymath}
		\underbrace{\fcon^H \gconplus} \gcon^\dagger + 
		(\fconplus)^H \gcon^\perp \left( \gcon^\perp \right)^H
		~=~  (\fconplus)^H \gcon \gcon^\dagger + 
		(\fconplus)^H \gcon^\perp \left( \gcon^\perp \right)^H
		~=~ (\fconplus)^H.
	\end{displaymath}
	Therefore, we have
	\begin{displaymath}
		\fcon^H	\mymat{A}
		~=~  (\fconplus)^H,
	\end{displaymath}
	and consequently,
	\begin{displaymath}
		\mathcal{A}~\supset~	\{\mymat{A}~|~\exists ~\tilde{\mymat{A}}~\textrm{such that}~ \mymat{A}~=~\optC_0 + \fcon^\perp \tilde{\mymat{A}} (\gcon^\perp)^H  \}.
	\end{displaymath}
\end{proof}

\begin{proof}[Proof of \cref{thm:Astar_IC_is_a_minimizer}]
	Suppose we construct the matrix \(\mymat{Q}_{\fcon} \) so that its columns form an orthonormal basis for \(\mathcal{R}(\fcon)\).
	Consider the ordinary least squares problem:
	\begin{equation}\label{eq:Reduced_OLS}
		\min_{\tilde{\mymat{A}}}~
		\left
		\lVert \tilde{\mymat{A}}  \left( \left(\gcon^\perp\right)^H \mymat{X} \right) - 
		\left(\fcon^\perp\right)^H 
		\left(\mymat{Y} - \optC_0 \mymat{X}\right)
		\right
		\rVert_F^2
		~+~
		\lVert
		\mymat{Q}_{\fcon}^H
		(\mymat{Y} - \optC_0 \mymat{X})
		\rVert_F^2.
	\end{equation}
	Observe that \( 	\optA_{\rm lsq} \) is defined in \cref{eq:defn_Astar_lsq} to be the (minimum-norm) optimizer of \cref{eq:Reduced_OLS}.
	We will see that \cref{eq:Reduced_OLS} can be connected to \cref{eq:IC_DMD} in a manner that lets \(\optA_{\rm IC}\) imbibe the optimality of \( 	\optA_{\rm lsq} \).
	
	We begin by factoring out \( \left( \fcon^\perp \right)^H \) from the first term and \(  \mymat{Q}_{\fcon}^H \)  from the second term in the objective of \cref{eq:Reduced_OLS}.
	Since \( \tilde{\mymat{A}} = \left( \fcon^\perp \right)^H \fcon^\perp \tilde{\mymat{A}}\), the first term reads thus:
	\begin{displaymath}
		\begin{aligned}
			&\left
			\lVert \tilde{\mymat{A}}  \left( \left(\gcon^\perp\right)^H \mymat{X} \right) - 
			\left(\fcon^\perp\right)^H 
			\left(\mymat{Y} - \optC_0 \mymat{X}\right)
			\right
			\rVert_F^2 \\
			~=~
			&\left
			\lVert 
			\left( \fcon^\perp \right)^H
			\left(
			\fcon^\perp
			\tilde{\mymat{A}}   \left(\gcon^\perp\right)^H \mymat{X}
			-
			\left(\mymat{Y} - \optC_0 \mymat{X}\right)
			\right)
			\right
			\rVert_F^2
		\end{aligned}
	\end{displaymath}
	Observe that replacing the pre-factor of \(\left( \fcon^\perp \right)^H\) above with \(\mymat{Q}_{\fcon}^H\) gives the second term of \cref{eq:Reduced_OLS}. 
	In other words,
	\begin{displaymath}
		\lVert
		\mymat{Q}_{\fcon}^H
		(\mymat{Y} - \optC_0 \mymat{X})
		\rVert_F^2
		~=~
		\left
		\lVert 
		\mymat{Q}_{\fcon}^H
		\left(
		\fcon^\perp
		\tilde{\mymat{A}}   \left(\gcon^\perp\right)^H \mymat{X}
		-
		\left(\mymat{Y} - \optC_0 \mymat{X}\right)
		\right)
		\right
		\rVert_F^2.
	\end{displaymath}
	Consequently, the objective function in \cref{eq:Reduced_OLS} may be re-written as follows:
	\begin{displaymath}
		\left\lVert
		\begin{bmatrix}
			\mymat{Q}_{\fcon}^H \\
			\left(\fcon^\perp\right)^H
		\end{bmatrix}
		\left(
		\fcon^\perp
		\tilde{\mymat{A}}   \left(\gcon^\perp\right)^H \mymat{X}
		-
		\left(\mymat{Y} - \optC_0 \mymat{X}\right)
		\right)
		\right\rVert_F^2
	\end{displaymath}
	By the definitions of \(\mymat{Q}_{\fcon}\) and \(\fcon^\perp\),
	\(
	\begin{bmatrix}
		\mymat{Q}_{\fcon} & \fcon^\perp
	\end{bmatrix}
	\) is a unitary matrix.
	Since the Frobenius norm is invariant under unitary operations, the above expression simplifies to
	\begin{displaymath}
		\left\lVert
		\fcon^\perp
		\tilde{\mymat{A}}   \left(\gcon^\perp\right)^H \mymat{X}
		-
		\left(\mymat{Y} - \optC_0 \mymat{X}\right)
		\right\rVert_F^2
		~=~
		\left\lVert
		\left(
		\optC_0
		+
		\fcon^\perp
		\tilde{\mymat{A}}   \left(\gcon^\perp\right)^H
		\right)
		\mymat{X}
		-
		\mymat{Y}
		\right\rVert_F^2.
	\end{displaymath}
	Since this is just a rephrasal of the cost function in \cref{eq:Reduced_OLS}, the optimality of \(	\optA_{\rm lsq}\) gives us the following inequality:
	\begin{displaymath}
		\forall~\tilde{\mymat{A}},~
		\left\lVert
		\left(
		\optC_0
		+
		\fcon^\perp
		\tilde{\mymat{A}}   \left(\gcon^\perp\right)^H
		\right)
		\mymat{X}
		-
		\mymat{Y}
		\right\rVert_F^2
		\geq
		\left\lVert
		\left(
		\optC_0
		+
		\fcon^\perp
		\optA_{\rm lsq}   \left(\gcon^\perp\right)^H
		\right)
		\mymat{X}
		-
		\mymat{Y}
		\right\rVert_F^2.
	\end{displaymath}
	Using  \cref{lem:Parametrization_mathcal_A}, we get
	\begin{displaymath}
		\forall~\mymat{A}\in \mathcal{A},~
		\lVert \mymat{A}~ \mymat{X} - \mymat{Y} \rVert_F^2~\geq~
		\left\lVert 
				\left(
		\optC_0
		+
		\fcon^\perp
		\optA_{\rm lsq}   \left(\gcon^\perp\right)^H
		\right)
		 ~\mymat{X} - \mymat{Y} \right\rVert_F^2.
	\end{displaymath}
	Since \cref{lem:Parametrization_mathcal_A} also says that
	\begin{displaymath}
				\left(
		\optC_0
		+
		\fcon^\perp
		\optA_{\rm lsq}   \left(\gcon^\perp\right)^H
		\right) \in \mathcal{A},
	\end{displaymath}
	 the above matrix is rendered a minimizer of \cref{eq:IC_DMD}.
\end{proof}

\begin{remark}
	The arguments made here are valid for a broader class of priors that only satisfies the weakened version of \cref{eq:Assumption_No_Redunancies} given below:
	\begin{equation}\label{eq:Assumption_Generalized_No_Redunancies}
		\gconplus \gcon^\dagger \gcon~=~\gconplus,\quad \fconplus \fcon^\dagger \fcon~=~\fconplus.
	\end{equation}
	The utility of this relaxation is that it accommodates the partial or complete absence of prior information. For example, we may know some invariant solutions of \cref{eq:DS4SDMD} but nothing about the associated Koopman eigen-functions. In this scenario, the matrices \(\gcon\) and \(\gconplus\) will be determined by the known invariant solutions while \(\fcon\) and \(\fconplus\) can be safely set to \(\mymat{0}\).
\end{remark}

\subsection{Establishing the recovery of Affine DMD}\label{ss:Proof__Affine_DMD__recovery}
\begin{proof}[Proof of \cref{prop:Affine_DMD__recovery} ]
	We organize the proofs similar to the statements.
	\begin{enumerate}
		\item According to \cref{eq:Constraints_for_Affine_DMD}, we have:
		\begin{displaymath}
			\begin{aligned}
				\mathcal{A} ~&=~ \{ \mymat{A} ~|~\fcon^H \mymat{A} ~=~\left(\fconplus\right)^H\} \\
				~&=~
				\left\{ 
				\mymat{A} ~|~\exists~ \mymat{A}_{21},~\mymat{A}_{22}~\textrm{such that}~
				\mymat{A}~=~
				\begin{bmatrix}
					1 & \\
					\mymat{A}_{21} & \mymat{A}_{22}
				\end{bmatrix}
				\right\}.
			\end{aligned}
		\end{displaymath}
		This gives the desired restriction on the first row of \(\optA_{\rm IC}\).
		\item Follows from the above, by using \cref{eq:IC_DMD} and \cref{eq:defn_Affine_DMD} simultaneously.
		\item The proof of Proposition 3.1 in \cite{hirsh2019centering} establishes \cref{eq:IC_DMD__offsets_analogoues_to__Affine_DMD} for real-valued \(\mymat{X}\) and \(\mymat{Y}\).
		Here, we generalize it\footnote{Primarily needed when \(\mathcal{R}\left(\rowselect{\mymat{Y}}_2^T\right) \nsubseteq \mathcal{R}\left(\mymat{X}^T\right)\), where the arguments in \cite{hirsh2019centering} would necessitate computing the gradient of the real-valued objective function with respect to a potentially \emph{complex-valued} \(\mymat{A}_{21}\). While this could be done in principle by completing the squares or representing complex numbers as \(2 \times 2\) matrices, we present a simpler alternative here.} to work for complex-valued matrices.	
		
		When \(\mathcal{R}\left(\rowselect{\mymat{Y}}_2^T\right) \subseteq \mathcal{R}\left(\mymat{X}^T\right)\), the optimal value of the \algonick~objective is 0. Hence, every solution must satisfy:
		\begin{displaymath}
			\left(\optA_{\rm IC}\right)_{22} \rowselect{\mymat{X}}_2 + \left(\optA_{\rm IC}\right)_{21} \myvec{1}_n^H = \rowselect{\mymat{Y}}_2.
		\end{displaymath}
		Post-multiplying both sides with \(\frac{\myvec{1}_\trainlength}{\trainlength}\) and then using \cref{eq:defn_Affine_DMD__Sample_means} yields \cref{eq:IC_DMD__offsets_analogoues_to__Affine_DMD}.
		
		On the other-hand, when \(\mathcal{R}\left(\rowselect{\mymat{Y}}_2^T\right) \nsubseteq \mathcal{R}\left(\mymat{X}^T\right)\), the convenient constraint from the preceding case is missing.
		As such, we are forced to work with the \algonick~objective \cref{eq:IC_DMD} which simplifies to the following relation due to the specific choices of \(\fcon~\textrm{and}~\fconplus\):
		\begin{displaymath}
			\left(\optA_{\rm IC}\right)_{22}, \left(\optA_{\rm IC}\right)_{21} ~=~ \arg\min_{\mymat{A}_{22}, \mymat{A}_{21}}~ \left\lVert \mymat{A}_{22} \rowselect{\mymat{X}}_2 + \mymat{A}_{21} \myvec{1}_n^H - \rowselect{\mymat{Y}}_2
			\right \rVert_{F}^2
		\end{displaymath}
		In order to decouple the two matrices, we can invoke the invariance of the Frobenius norm under unitary transformations, to re-write the objective function above as,
		\begin{displaymath}
			\left\lVert  \left(\mymat{A}_{22} \rowselect{\mymat{X}}_2 + \mymat{A}_{21} \myvec{1}_n^H - \rowselect{\mymat{Y}}_2 \right) \mymat{Q}
			\right \rVert_{F}^2,
		\end{displaymath}
		where \(\mymat{Q}\) is the following unitary matrix:
		\begin{displaymath}
			\mymat{Q} ~:=~ 
			\begin{bmatrix}
				\frac{\myvec{1}_\trainlength}{\sqrt{\trainlength}} & \mymat{Q}_2
			\end{bmatrix}.
		\end{displaymath}
		The first column of the matrix product can be written in-terms of \(\myvec{\mu}_{\mymat{X}}~\textrm{and}~\myvec{\mu}_{\mymat{Y}}\) using \cref{eq:defn_Affine_DMD__Sample_means}:
		\begin{displaymath}
			\left(\mymat{A}_{22} \rowselect{\mymat{X}}_2 + \mymat{A}_{21} \myvec{1}_n^H - \rowselect{\mymat{Y}}_2 \right)
			\frac{\myvec{1}_\trainlength}{\sqrt{\trainlength}}
			~=~
			\sqrt{\trainlength}
			\left(
			\mymat{A}_{22} \myvec{\mu}_{\mymat{X}} + \mymat{A}_{21} - \myvec{\mu}_{\mymat{Y}}
			\right)
		\end{displaymath}
		The remaining columns of the matrix product do not possess \(\mymat{A}_{21} \) due to the unitary nature of \(\mymat{Q}\):
		\begin{displaymath}
			\left(\mymat{A}_{22} \rowselect{\mymat{X}}_2 + \mymat{A}_{21} \underbrace{\myvec{1}_n^H} - \rowselect{\mymat{Y}}_2 \right)
			\underbrace{
				\mymat{Q}_2
			}
			~=~
			\left(
			\mymat{A}_{22} \rowselect{\mymat{X}}_2  - \rowselect{\mymat{Y}}_2
			\right)
			\mymat{Q}_2
		\end{displaymath}
		Putting these together, the de-coupled \algonick~ objective reads thus:
		\begin{displaymath}
			\begin{aligned}
				&\left(\optA_{\rm IC}\right)_{22}, \left(\optA_{\rm IC}\right)_{21} ~=~ \\ &\arg\min_{\mymat{A}_{22}, \mymat{A}_{21}}~ 
				\trainlength
				\left\lVert \mymat{A}_{22} \myvec{\mu}_{\mymat{X}} + \mymat{A}_{21} - \myvec{\mu}_{\mymat{Y}} 
				\right \rVert_2^2
				+
				\left \lVert
				\left(
				\mymat{A}_{22} \rowselect{\mymat{X}}_2  - \rowselect{\mymat{Y}}_2
				\right)
				\mymat{Q}_2
				\right \rVert_F^2
			\end{aligned}
		\end{displaymath}
		Since the first term can be eliminated using \(\mymat{A}_{21}\), independent of the choice of \(\mymat{A}_{22}\), we get,
		\begin{displaymath}
			\left( \optA_{\rm IC}  \right)_{22} \myvec{\mu}_{\mymat{X}} + \left(\optA_{\rm IC} \right)_{21} - \myvec{\mu}_{\mymat{Y}}  = \myvec{0},
		\end{displaymath}
		which is the desired conclusion. 
	\end{enumerate}
\end{proof}

\bibliographystyle{siamplain}
\bibliography{ic_dmd}

\end{document}

%% file: ota2invsets_shared.tex
\usepackage{lipsum}
\usepackage{amsfonts}
\usepackage{amssymb}
\usepackage{graphicx}
\usepackage{epstopdf}
\usepackage{algorithmic}
\usepackage{stmaryrd}
\ifpdf
  \DeclareGraphicsExtensions{.eps,.pdf,.png,.jpg}
\else
  \DeclareGraphicsExtensions{.eps}
\fi

\graphicspath{{plots/}{./ic_plots/}}

\newsiamremark{remark}{Remark}
\newsiamremark{hypothesis}{Hypothesis}
\crefname{hypothesis}{Hypothesis}{Hypotheses}
\newsiamthm{claim}{Claim}

\newcommand{\myvec}[1]{\boldsymbol{\mathbf{#1}}}
\newcommand{\mymat}[1]{\mathbf{#1}}

\newcommand{\algoname}{\textrm{Invariant Consistent DMD}}
\newcommand{\algonick}{\textrm{IC-DMD}}
\newcommand{\state}{\myvec{\upsilon}}
\newcommand{\dictionary}{\myvec{\psi}}

\newcommand{\undelayed}[1]{\breve{#1}}
\newcommand{\undelayedictionary}{\undelayed{\dictionary}}

\newcommand{\ssdim}{p}
\newcommand{\dictionarylength}{m}
\newcommand{\undelayedictlength}{\undelayed{\dictionarylength}}
\newcommand{\trainlength}{n}
\newcommand{\geomconscount}{g}
\newcommand{\funconscount}{f}
\newcommand{\fundelayscount}{d}

\newcommand{\genericcount}{o}

\newcommand{\gcon}{\mymat{D}}
\newcommand{\fcon}{\mymat{E}}
\newcommand{\fconvec}{\myvec{e}}

\newcommand{\gconplus}{\gcon^{\oplus}}
\newcommand{\fconplus}{\fcon^{\oplus}}
\newcommand{\fconplusvec}{\myvec{e}^{\oplus}}

\newcommand{\gconequalizer}{\widehat{\mymat{V}}}

\newcommand{\rowselect}[1]{\left\llbracket #1 \right\rrbracket}

\newcommand{\opt}[1]{#1^*}

\newcommand{\optA}{\opt{\mymat{A}}}
\newcommand{\optC}{\opt{\mymat{C}}}

\newcommand{\reduced}[1]{\widetilde{#1}}

\newcommand{\reducedopt}[1]{\opt{\reduced{#1}}}

\newcommand{\reducedoptA}{\reducedopt{\mymat{A}}}

\newcommand{\dummy}[1]{\widehat{#1}}

\newcommand{\DMD}{Dynamic Mode Decomposition~}

\newcommand{\naiveDMDFlavour}{Extended~}

\newcommand{\naiveDMDAcronym}{EDMD}
\newcommand{\naiveDMDMatrix}{\mymat{A}_{\rm \naiveDMDAcronym}}

\newcommand{\velocity}{\myvec{f}}
\newcommand{\flow}{\mathcal{S}}

\usepackage{hyperref}
\usepackage{easyReview}
\usepackage[caption=false]{subfig}

\title{Invariant Consistent Dynamic Mode Decomposition}

\author{Gowtham S Seenivasaharagavan\thanks{Department of Mechanical Engineering,
		University of California Santa Barbara} \and Milan Korda\thanks{Methods and Algorithms for Control,	LAAS-CNRS} \and Hassan Arbabi \thanks{Department of Mechanical Engineering,	Massachusetts Institute of Technology} \and Igor Mezi{\'c}\thanks{Department of Mechanical Engineering,
		University of California Santa Barbara}}

\usepackage{amsopn}
